\documentclass[oneside,english]{amsart}
\usepackage[T1]{fontenc}
\usepackage[latin9]{inputenc}
\usepackage{textcomp}
\usepackage{amstext}
\usepackage{amsthm}
\usepackage{amssymb}

\makeatletter
\theoremstyle{plain}
\newtheorem{thm}{\protect\theoremname}
\theoremstyle{definition}
\newtheorem{defn}[thm]{\protect\definitionname}
\theoremstyle{plain}
\newtheorem{prop}[thm]{\protect\propositionname}
\theoremstyle{plain}
\newtheorem{lem}[thm]{\protect\lemmaname}
\theoremstyle{plain}
\newtheorem{cor}[thm]{\protect\corollaryname}
\theoremstyle{remark}
\newtheorem*{acknowledgement*}{\protect\acknowledgementname}

\usepackage{alex}

\makeatother

\usepackage{babel}
\providecommand{\acknowledgementname}{Acknowledgement}
\providecommand{\corollaryname}{Corollary}
\providecommand{\definitionname}{Definition}
\providecommand{\lemmaname}{Lemma}
\providecommand{\propositionname}{Proposition}
\providecommand{\theoremname}{Theorem}

\begin{document}
\title[Finiteness of the class group]{A uniform proof of the finiteness of the class group of a global field}
\author{Alexander Stasinski}
\address{Department of Mathematical Sciences, Durham University, Durham, DH1
3LE, UK}
\email{alexander.stasinski@durham.ac.uk}
\begin{abstract}
We give a definition of a class of Dedekind domains which includes
the rings of integers of global fields and give a proof that all rings
in this class have finite ideal class group. We also prove that this
class coincides with the class of rings of integers of global fields.
\end{abstract}

\maketitle

\section{Introduction.}

\subsection{Background.}

The starting point of algebraic number theory is to define an (algebraic)
\emph{number field} $K$ as a finite field extension of $\Q$ and
the ring $\Z_{K}$ of \emph{algebraic integers} in $K$ as all the
elements in $K$ that satisfy a monic polynomial equation with coefficients
in $\Z$ (This is called the \emph{integral closure} of $\Z$ in $K$.)
Unlike the situation in $\Z$, unique factorization of elements in
$\Z_{K}$ into irreducibles can fail. Nevertheless, $\Z_{K}$ is an
example of a \emph{Dedekind domain}, that is, an integral domain in
which every nonzero proper ideal is, uniquely, a product of prime
ideals.

One can develop a parallel theory of finite extensions $K$, called
(algebraic) \emph{function fields}, of the field $\F_{q}(t)$, which
is the field of fractions of the ring of polynomials $\F_{q}[t]$
over a finite field $\F_{q}$. The analogue of $\Z_{K}$ is then the
ring of elements in $K$ that satisfy a monic polynomial equation
with coefficients in $\F_{q}[t]$. These rings are also Dedekind domains
and their theory can to a large extent be developed in parallel with
that of the rings $\Z_{K}$. For this reason, it is sometimes convenient
to use the term \emph{global field} for either a number field or a
function field.

One of the most fundamental problems about rings of integers in global
fields is the study of the (failure of) unique factorization. This
is encoded in the \emph{(ideal) class group} of the ring, which is
trivial if and only if unique factorization holds. A important result
in algebraic number theory is that the class group of $\Z_{K}$ is
finite. Similarly, it is known that the class group of a ring of integers
in a function field is finite. One can define the ideal class group
$Cl(R)$ of any Dedekind domain $R$ as the equivalence classes of
nonzero ideals where two ideals $I,J$ are said to be equivalent if
$aI=bJ$, for some nonzero $a,b\in R$, and the group operation is
induced by multiplication of ideals.

\subsection{The main results.}

In \cite{Clark-MO-question}, P.~L.~Clark asked whether there exists
a ``purely algebraic'' proof of the finiteness of the class group
of global fields and whether there exist any ``structural'' conditions
on a Dedekind domain that imply the finiteness of its class group.

In this article we answer these questions in the affirmative. More
precisely, we introduce a class (G) of Dedekind domains (see Definition~\ref{def:Gen-num-ring})
that contains the rings of integers of global fields. We then give
a uniform (i.e., not case by case) proof that any ring of class (G)
has finite ideal class group. We also give a known argument showing
how to deduce the finiteness of the class group of any overring of
a ring of class (G). 

The proof of the finiteness of the class group that we give is essentially
that of R.~Swan \cite[Theorem~3.9]{Swan-book} and I.~Reiner \cite[(26.3)]{Reiner-max-ord}
(modulo some exercises) for rings of integers in global fields. The
contribution here is that we axiomatize properties of a Dedekind domain
sufficient for the proof to go through and that we show that these
properties in fact characterize rings of integers in global fields.
This shows in particular that the finiteness of the class groups of
global fields can be proved uniformly without adeles and without methods
from the geometry of numbers. This fact does not seem to have been
widely known. Indeed, Clark writes in \cite{Clark-MO-question} that
``[\ldots] it is generally held that the finiteness of the class
number is one of the first results of algebraic number theory which
is truly number-theoretic in nature and not part of the general study
of commutative rings'' and in \cite[B.1, p.~334]{Ramakrishnan-Valenza}
the authors write: ``Note well that for a general Dedekind domain,
$Cl_{K}$ need not be finite. This shows that one essentially needs
some analysis to supplement the abstract algebra in Chapter 5.'' 

A key idea of the proof is to estimate the norm of an element from
above algebraically using the fact that a determinant is a homogeneous
polynomial in the entries of a matrix (see the proof of Lemma~\ref{lem:Norm-polynomial}).
This idea is present in \cite[p.~53]{Swan-book}, \cite[(26.3)]{Reiner-max-ord}
and \cite[(20.10)]{Curtis_Reiner} but can be traced back to Zassenhaus
\cite{Zassenhaus-class-number} in the number field case and Higman--McLaughlin
\cite{Higman-McLaughlin} in the function field case. By contrast,
the standard nonadelic and nongeometric proof of the finiteness of
the class group in the number field case (see, e.g., \cite[V, Section~4]{Lorenzini})
expresses the field norm in terms of the complex absolute values of
Galois conjugates, and in the function field case this needs a  modification
involving absolute values. 

In the final section, we show that the class (G) coincides with the
class of global fields. This uses the Artin--Whaples axiomatization
of global fields and shows that the quasi-triangle inequality condition
in Definition~\ref{def:Gen-num-ring}, despite its simplicity and
elementary nature, implies the product formula for absolute values.

\section{Basic PIDs and rings of class (G).}

All rings are commutative with identity. Let $\N$ denote the set
of positive integers. We use the standard acronym PID for ``principal
ideal domain.'' A ring $R$ is called a \emph{finite quotient domain}
or is said to\emph{ have finite quotients} if for every nonzero ideal
$I$ of $R$, the quotient $R/I$ is a finite ring. If $R$ is a finite
quotient domain, $I\subseteq R$ a nonzero ideal, and $x\in R$ is
nonzero, we write $N_{R}(I)=|R/I|$ and $N_{R}(x)=|R/xR|$. We also
define $N_{R}(0)=0$. The function $N_{R}:R\rightarrow\N\cup\{0\}$
is called the \emph{ideal norm} on $R$. It is known that if $R$
is a finite quotient Dedekind domain, then $N_{R}$ is multiplicative
(see \cite[Lemma~V.3.5]{Lorenzini}).
\begin{defn}
\label{def:Basic PID}We call a PID $A$ a \emph{basic PID} if it
is not a field and if the following conditions are satisfied: 
\begin{enumerate}
\item $A$ is a finite quotient domain;
\item there exists a constant $c\in\N$ such that for each $m\in\N$,
\[
\#\{x\in A\mid N_{A}(x)\leq c\cdot m\}\geq m
\]
(i.e., $A$ has ``enough elements of small norm'');
\item there exists a constant $C\in\N$ such that for all $x,y\in A$,
\[
N_{A}(x+y)\leq C\cdot(N_{A}(x)+N_{A}(y))
\]
(i.e., $N_{A}$ satisfies the ``quasi-triangle inequality'').
\end{enumerate}
\end{defn}

There exist PIDs for which the first and second conditions in Definition~\ref{def:Basic PID}
hold but for which the third condition fails. Take, for instance,
the PID $A=\Z[\sqrt{2}]$. Then $u=1+\sqrt{2}$ is a unit in $A$
and $\bar{u}=1-\sqrt{2}$. For any $r\in\N$ write $u^{r}=a_{r}+b_{r}\sqrt{2}$,
for $a_{r},b_{r}\in\Z$. Then $a_{r}$ grows with $r$ and
\[
N_{A}(u^{r}+\bar{u}^{r})=N(2a_{r})=|N_{\Q(\sqrt{2})/\Q}(2a_{r})|=4a_{r}^{2},
\]
while $N_{A}(u^{r})+N_{A}(\bar{u}^{r})=2$, since $u^{r}$ and $\bar{u}^{r}$
are units. Thus the third condition in Definition~\ref{def:Basic PID}
fails for $A$ even though $A$ is a finite quotient domain that satisfies
the second condition since it has infinitely many units. 

Another example of a finite quotient PID $A$ where the second condition
holds but the third condition fails is the localization $\Z_{(p)}$
of $\Z$ at a prime $p$, that is, the subring of $\Q$ consisting
of fractions $a/b$, $a,b\in\Z$, where $p\nmid b$. Here $\pm1+p^{n}$
is a unit for every $n\in\N$, so $N_{A}(1+p^{n})+N_{A}(-1+p^{n})=1+1=2$,
while $N_{A}(1+p^{n}+(-1+p^{n}))=N_{A}(2)N_{A}(p)^{n}$, which grows
with $n$.
\begin{defn}
\label{def:Gen-num-ring}Let $A$ be a basic PID. We call a Dedekind
domain $B$ a\emph{ ring of class (G) (over $A$)} if $B$ is an $A$-algebra
that is finitely generated and free as a module over $A$.
\end{defn}

Since free modules over a PID are torsion-free, we may and will consider
$A$ as a subring of $B$ via the embedding $a\mapsto a\cdot1$. Our
goal in the next section is to prove that any ring $B$ of class (G)
has finite ideal class group. The terminology ``(G)'' is provisional
(``G'' for global), because it will turn out that the class (G)
is equal to the class of rings of integers in global fields (see Corollary~\ref{cor:class-(G)-is-global-rings}).
\begin{defn}
By \emph{global field} we mean either a finite extension of $\Q$
or a finite separable extension of some $\F_{q}(t)$, where $t$ is
transcendental over $\F_{q}$. By a \emph{ring of integers} of a global
field $K$ we mean either the integral closure in $K$ of $\Z$ (in
the number field case) or the integral closure in $K$ of $\F_{q}[t]$,
for some $t\in K$ transcendental over $\F_{q}$ (in the function
field case).
\end{defn}

Note that in the function field case, there is no unique ring of integers,
as, for instance, one can also take the integral closure of $\F_{q}[t^{-1}]$.
\begin{prop}
\label{prop:Global fields}Let $B$ be a ring of integers of a global
field. Then $B$ is a ring of class (G) over $\Z$ or $\F_{q}[t]$,
respectively.
\end{prop}

\begin{proof}
First, it is straightforward to check that $\Z$ is a basic PID. Indeed,
all its proper quotients $\text{\Z}/n$ are finite, $N_{\Z}(n)=|n|$
(the absolute value of $n$) so $\#\{x\in\Z\mid N_{\Z}(x)\leq m\}=2m+1\geq m$
and $N_{\Z}(x+y)=|x+y|\leq|x|+|y|=N_{\Z}(x)+N_{\Z}(y)$, thanks to
the usual triangle inequality. 

Next, let $A=\F_{q}[t]$. For any $f(t)\in A$ we have $N_{A}(f(t))=q^{\deg(f)}$,
so $A$ is a finite quotient domain and 
\begin{align*}
\#\{x\in A\mid N_{A}(x)\leq m\} & =\#\{x\in A\mid\deg(x)\leq\floor{\log_{q}(m)}\}\\
 & =q^{\floor{\log_{q}(m)}+1}>q^{\log_{q}(m)}=m,
\end{align*}
so the second property in Definition~\ref{def:Basic PID} is satisfied
for $A$. Furthermore, for $f(t),g(t)\in A$ we have
\begin{align*}
N_{A}(f(t)+g(t)) & =q^{\deg(f+g)}=q^{\max\{\deg f,\deg g\}}\leq q^{\deg f}+q^{\deg g}\\
 & =N_{A}(f(t))+N_{A}(g(t)),
\end{align*}
so $A$ is a basic PID. Thus both $\Z$ and $\F_{q}[t]$ satisfy Definition~\ref{def:Basic PID}
with $c=C=1$.

It is well known that $B$ is a Dedekind domain and is free of finite
rank over $\Z$ or $\F_{q}[t]$, respectively (see \cite[Theorem~I.4.7]{Lorenzini}
for the number field case and \cite[Theorem~X.1.7]{Lorenzini} for
the function field case; note that if the extension of fraction fields
is finite and separable, which is always the case for number fields,
this follows with a classical proof, but for function fields, where
separability may fail, it requires a separate proof.). Thus, in either
case, $B$ is a ring of class (G).
\end{proof}

\section{Norm estimates and finiteness.}

Throughout this section, let $B$ be a ring of class (G) over the
basic PID $A$ and let $K$ and $L$ be the field of fractions of
$A$ and $B$, respectively.

For $\alpha\in L$ we let $T_{\alpha}$ denote the endomorphism $L\rightarrow L$,
$x\mapsto\alpha x$, and define the norm $N_{L/K}(\alpha)=\det(T_{\alpha})$.
As is well known, the fact that $B$ is the integral closure of $A$
in $L$ implies that $N_{L/K}(B)\subseteq A$ (see, e.g., \cite[Corollary~IV.2.4]{Lorenzini}).

The following lemma is a consequence of \cite[Proposition IV.6.9 and Proposition~V.3.6]{Lorenzini}
(which is valid when $A$ is a Dedekind domain, not necessarily a
PID). We give a simple proof in our setting (where $A$ is a PID),
exploiting the Smith normal form (see, e.g., \cite[Section~5.3]{Adkins-Weintraub}).
\begin{lem}
\label{lem:N(alphaB)-N(N(alpha)A)}For any nonzero $\alpha\in B$,
we have $N_{B}(\alpha)=N_{A}(N_{L/K}(\alpha))$.
\end{lem}

\begin{proof}
We have $N_{B}(\alpha)=|B/\alpha B|$ and $B/\alpha B$ is the cokernel
of the map $T_{\alpha}:B\rightarrow B$. By the Smith normal form,
we have
\[
B/\alpha B\cong A/p_{1}A\oplus\dots\oplus A/p_{n}A,
\]
where $n$ is the rank of $B$ over $A$, and $p_{i}\in A$ are some
nonunits such that $\det(T_{\alpha})=u^{-1}p_{1}\cdots p_{n}$, for
some unit $u\in A$ ($u=\det(PQ)$ where $PT_{\alpha}Q$ is the Smith
normal form, with $T_{\alpha}$ identified with its matrix with respect
to some chosen basis).

Now observe that for any $m_{1},\dots,m_{k}\in A$, we have 
\[
|A/m_{1}\cdots m_{k}A|=|A/m_{1}A|\cdots|A/m_{k}A|,
\]
which follows from the Chinese remainder theorem (see, e.g., \cite[Corollary~2.25]{Adkins-Weintraub}),
combined with the fact that for any irreducible element $m\in A$
and $i\in\N$, we have $|m^{i}A/m^{i+1}A|=|A/mA|$ (the map $A\to m^{i}A$
given by $1\mapsto m^{i}$ induces an isomorphism $A/mA\rightarrow m^{i}A/m^{i+1}A$).

Thus 
\begin{align*}
N_{A}(N_{L/K}(\alpha)) & =|A/N_{L/K}(\alpha)A|=|A/\det(T_{\alpha})A|=|A/p_{1}A|\cdots|A/p_{n}A|\\
 & =|B/\alpha B|=N_{B}(\alpha).
\end{align*}
\end{proof}
\begin{lem}
\label{lem:Norm-polynomial}Let $x_{1},\dots,x_{n}$ be a basis for
$B$ over $A$. Let $\alpha\in B$ and write $\alpha=c_{1}x_{1}+\dots+c_{n}x_{n}$,
with $c_{i}\in A$. Then there exists a homogeneous polynomial $f(T_{1},\dots,T_{n})$
over $A$ of degree $n$ such that
\[
N_{L/K}(\alpha)=f(c_{1},\dots,c_{n}).
\]
Moreover, there exists a constant $C\in\N$ such that
\[
N_{B}(\alpha)\leq C\cdot\max_{i}\{N_{A}(c_{i})\}^{n}.
\]
\end{lem}

\begin{proof}
For $1\leq i,j,k\leq n$, let $r_{ij}^{(k)}\in A$ be such that 
\[
x_{i}x_{j}=\sum_{k=1}^{n}r_{ij}^{(k)}x_{k}.
\]
Then 
\begin{align*}
\alpha x_{i} & =c_{1}x_{i}x_{1}+\cdots+c_{n}x_{i}x_{n}=c_{1}\sum_{k=1}^{n}r_{i1}^{(k)}x_{k}+\cdots+c_{n}\sum_{k=1}^{n}r_{in}^{(k)}x_{k}\\
 & =\sum_{k=1}^{n}\Big(\sum_{j=1}^{n}c_{j}r_{ij}^{(k)}\Big)x_{k},
\end{align*}
so the matrix of $T_{\alpha}$ with respect to the basis $x_{1},\dots,x_{n}$
has $(i,k)$-entry equal to $\sum_{j=1}^{n}c_{j}r_{ij}^{(k)}$, for
$1\leq i,k\leq n$. Hence each entry of the matrix of $T_{\alpha}$
is a linear form in $c_{1},\dots,c_{n}$, and therefore $\det(T_{\alpha})=f(c_{1},\dots,c_{n})$
for some homogeneous polynomial $f$ of degree $n$.

Moreover, write $f(c_{1},\dots,c_{n})=a_{1}c_{1}^{n_{1,1}}\cdots c_{n}^{n_{1,n}}+\dots+a_{k}c_{1}^{n_{k,1}}\cdots c_{n}^{n_{k,n}}$,
where $a_{i}\in A$, $k,n_{i,j}\in\N$ and $\sum_{j=1}^{n}n_{i,j}=n$,
for every $i$. By Lemma~\ref{lem:N(alphaB)-N(N(alpha)A)} and the
quasi-triangle inequality for $N_{A}$, there exists a constant $C_{0}\in\N$
such that
\begin{align*}
N_{B}(\alpha) & =N_{A}(N_{L/K}(\alpha))=N_{A}(f(c_{1},\dots,c_{n}))\\
 & \leq C_{0}\big(N_{A}(a_{1})N_{A}(c_{1})^{n_{1,1}}\cdots N_{A}(c_{n})^{n_{1,n}}+\cdots\\
 & \quad+N_{A}(a_{k})N_{A}(c_{1})^{n_{k,1}}\cdots N_{A}(c_{n})^{n_{k,n}}\big)\\
 & \leq C_{0}k\cdot\max_{i}\{N_{A}(a_{i})\}(\max_{i}\{N_{A}(c_{i})\})^{n}.
\end{align*}
Hence the result follows by letting $C=C_{0}k\cdot\max_{i}\{N_{A}(a_{i})\}$.
\end{proof}
\begin{thm}
\label{thm:Bound-finiteness}Suppose that $B$ is a ring of class
(G) over $A$. Then there exists a constant $C\in\N$ such that for
any ideal $I$ in $B$, there exists a nonzero element $\alpha\in I$
such that
\[
N_{B}(\alpha)\leq C\cdot N_{B}(I).
\]
Hence the ideal class group of $B$ is finite.
\end{thm}

\begin{proof}
Let $x_{1},\dots,x_{n}$ be a basis for $B$ over $A$. Let $m$ be
the unique positive integer such that $m^{n}\leq N_{B}(I)<(m+1)^{n}$.
The fact that $A$ is a basic PID (the second property) says that
there exists a $c\in\N$ such that for every $m$, $\#\{x\in A\mid N_{A}(x)\leq cm\}\geq m$.
Thus, for every $m$, the set
\[
S_{m}:=\{x\in A\mid N_{A}(x)\leq2cm\}
\]
has at least $m+1$ elements. Hence the set 
\[
S_{m}x_{1}+\dots+S_{m}x_{n}
\]
has at least $(m+1)^{n}$ distinct elements. Since $(m+1)^{n}>|B/I|$,
there exist two distinct elements $s$ and $t$ in the set $S_{m}x_{1}+\dots+S_{m}x_{n}$
that are congruent mod $I$. Write $s=\sum_{i=1}^{n}a_{i}x_{i}$ and
$t=\sum_{i=1}^{n}b_{i}x_{i}$, with $a_{i},b_{i}\in S_{m}$. Then
\[
s-t=\sum_{i=1}^{n}(a_{i}-b_{i})x_{i}
\]
is a nonzero element of $I$ and by the third property of Definition~\ref{def:Basic PID},
there is a $C_{0}\in\N$ such that
\[
N_{A}(a_{i}-b_{i})\leq C_{0}(N_{A}(a_{i})+N_{A}(b_{i}))\leq C_{0}2\cdot2cm.
\]
Thus Lemma~\ref{lem:Norm-polynomial} implies that there is a $C_{1}\in\N$
such that 
\[
N_{B}(s-t)\leq C_{1}\cdot\max_{i}\{N_{A}(a_{i}-b_{i})\}^{n}\leq C_{1}(C_{0}4cm)^{n},
\]
and thus 
\[
\frac{N_{B}(s-t)}{N_{B}(I)}\leq\frac{C_{1}(C_{0}4cm)^{n}}{m^{n}}=C_{1}(C_{0}4c)^{n}.
\]
Taking $\alpha=s-t$ and $C=C_{1}(C_{0}4c)^{n}$ thus proves the first
assertion of the theorem.

A well-known argument now implies the finiteness of the class group
of $B$ (see, e.g., \cite[Lemmas~V.3.8--3.9]{Lorenzini}). We give
the argument here for the convenience of the reader. If $I$ is an
ideal of $B$, we write $[I]$ for the corresponding ideal class in
the class group $Cl(B)$. We will first show that any ideal class
$\mfc\in Cl(R)$ contains an ideal $I$ such that $N(I)\leq C$. Let
$J$ be an ideal of $B$ such that $\mfc=[J]$. 

By the first assertion of the theorem, there exists a nonzero $\alpha\in J$
and a $C\in\N$ such that $N(\alpha)\leq C\cdot N(J)$. Since $\alpha B\subseteq J$,
the unique factorization of ideals in $B$ implies that we have $\alpha B=IJ$,
for some ideal $I$. Since $[\alpha B]$ is the trivial ideal class,
\[
[I]=[J]^{-1}=\mfc^{-1},
\]
and by the multiplicativity of $N$,
\[
N(J)N(I)=N(\alpha)\leq C\cdot N(J),
\]
so $N(I)\leq C$. We have thus shown what we wanted for $\mfc^{-1}$.
But $\mfc\in Cl(B)$ was arbitrary, so it holds for all $\mfc$. Now,
since there are only finitely many ideals of norm below a given bound
(see, e.g., \cite[Lemma~V.3.7]{Lorenzini}) we conclude that there
can only be finitely many classes $\mfc\in Cl(B)$.
\end{proof}
The theorem above together with Proposition~\ref{prop:Global fields}
imply that rings of integers of global fields have finite ideal class
group.

Let $D$ be an integral domain with field of fractions $K$. A ring
$R$ such that $D\subseteq R\subseteq K$ is called an \emph{overring}
of $D$. The following is a known result.
\begin{lem}
Let $D$ be a Dedekind domain with finite class group. Then any overring
$R$ of $D$ is a Dedekind domain with finite class group.
\end{lem}

\begin{proof}
It is well known that $R$ is a Dedekind domain (see, e.g., \cite[Lemma~1-1]{Claborn-65}).
Since the class group of $D$ is finite, hence torsion, a result independently
due to Davis \cite[Theorem~2]{Davis-64}, Gilmer and Ohm \cite[Cor.~2.6]{Gilmer-Ohm},
and Goldman \cite[§1, Corollary~(1)]{Goldman} implies that $R$ is
the localization of $D$ at a multiplicative subset of $D$. Then,
by a straightforward argument (see \cite[Proposition~1-2, Corollary~1-3]{Claborn-65}),
the class group of $R$ is a quotient of the class group of $D$,
hence is finite.
\end{proof}
The class of overrings of rings of class (G) includes all $S$-integer
rings, for any finite set $S$ of places containing the Archimedean
ones. On the other hand, by Theorem~\ref{thm:BasicPIDs-are-rings-of-integers}
it will follow that a ring of $S$-integers is not of type (G) unless
it is a ring of integers of a global field. 

\section{Rings of class (G) and global fields.}

For the reader's convenience, we state a few definitions and results
from Artin's book \cite[Chapter~1]{Artin-Alg_Num_and_Func}.
\begin{defn}
\label{def:An-absolute-value}An \emph{absolute value} (called ``valuation''
in \cite[Chapter~1]{Artin-Alg_Num_and_Func}) of a field $K$ is a
function $|\cdot|:K\to\R$, $x\mapsto|x|$, satisfying the following
conditions:
\begin{enumerate}
\item $|x|\geq0$ and $|x|=0$ if and only if $x=0$;
\item $|xy|=|x|\cdot|y|$;
\item there exists a constant $c\in\R$, $c\geq1$ such that if $|x|\leq1$,
then $|1+x|\leq c$.
\end{enumerate}
\end{defn}

Note that the third condition is equivalent to $|\cdot|$ satisfying
the quasi-triangle inequality (cf.~the third condition in Definition~\ref{def:Basic PID}).
Indeed, let $c$ be as in the third condition above and let $x,y\in K$.
If either $x=0$ or $y=0$, the quasi-triangle inequality is trivially
satisfied, so we may assume that $x\neq0$, $y\neq0$, and without
loss of generality $|x/y|\leq1$. Then $|1+x/y|\leq c\leq2c(1+|x/y|)$,
so $|x+y|\leq C(|x|+|y|)$, with $C=2c$. Conversely, if $C\in\R$
is a positive number such that $|x+y|\leq C(|x|+|y|)$ holds for all
$x,y\in K$, then in particular $|1+x|\leq C(1+|x|)$, and by making
$C$ larger if necessary, we can take $C\geq1$. Thus, if $|x|\leq1$,
we obtain $|1+x|\leq2C$, so the third condition in Definition~\ref{def:An-absolute-value}
holds with $c=2C$.

The trivial absolute value is the one for which $|x|=1$ for all nonzero
$x\in K$. One defines two absolute values $|\cdot|_{1}$ and $|\cdot|_{2}$
to be \emph{equivalent} if for any $x\in K$, $|a|_{1}<1$ if and
only if $|\cdot|_{2}<1$. It turns out that every absolute value is
equivalent to one for which the usual triangle inequality holds.

Let $K$ be a field and $|\cdot|_{v}$ and absolute value of $K$.
The absolute value $|\cdot|_{v}$ is said to be \emph{non-Archimedean}
if for all $x,y\in K$,
\[
|x+y|_{v}\leq\max\{|x|_{v},|y|_{v}\};
\]
otherwise $|\cdot|_{v}$ is said to be \emph{Archimedean}. We call
$|\cdot|_{v}$ \emph{discrete }if $|K|_{v}$ is a discrete subset
of $\R$. If $|\cdot|_{v}$ is non-Archimedean, $\mathcal{\cO}_{v}=\{x\in K\mid|x|_{v}\leq1\}$
is a ring (called the \emph{valuation ring} at $v$), $\mfp_{v}=\{x\in K\mid|x|_{v}<1\}$
is a maximal ideal of $\mathcal{\cO}_{v}$, and the field
\[
\bar{k}_{v}=\cO_{v}/\mfp_{v}
\]
is called the \emph{residue class field} at $v$. 

Let $\Sigma$ be a set of nonequivalent and nontrivial absolute values
of $K$. Consider the set $k_{0}=\{x\in K\mid|x|_{v}\leq1\text{ for all }v\in\Sigma\}$.
It is not hard to show that $k_{0}$ is a field if and only if $\Sigma$
contains no Archimedean prime (see \cite[Chapter~12, Section~1]{Artin-Alg_Num_and_Func}).
In this case we may consider $k_{0}$ as a subfield of each $\bar{k}_{v}$.

We will use the following fundamental result, due to Artin and Whaples
\cite{Artin-Whaples}.
\begin{thm}
\label{thm:Artin--Whaples}Suppose that $K$ is a field with a set
$\Sigma$ of mutually nonequivalent and nontrivial absolute values
such that the following two conditions hold:
\begin{enumerate}
\item For every $x\in K^{\times}$, $|x|_{v}=1$ for all but a finite number
of $v\in\Sigma$ and
\[
\prod_{v\in\Sigma}|x|_{v}=1;
\]
\item there is at least one $v\in\Sigma$ such that either $v$ is Archimedean
or $v$ is discrete and $\bar{k}_{v}$ is finite.
\end{enumerate}
Then $K$ is a global field. 
\end{thm}

A comment on the proof of this theorem. The proof of \cite[Chapter~12, Theorem~3]{Artin-Alg_Num_and_Func}
shows that under the conditions of Theorem~\ref{thm:Artin--Whaples},
if $\Sigma$ has at least one Archimedean absolute value, then $K$
is a number field and otherwise $K$ is a finite extension of $k_{0}(t)$,
for some $t$ transcendental over $k_{0}$. In the latter case, $k_{0}$
is a subfield of the finite field $\bar{k}_{v}$, so $k_{0}$ itself
is finite and thus $K$ is a global function field. 

We now come to the main result of the present section.
\begin{thm}
\label{thm:BasicPIDs-are-rings-of-integers}Let $A$ be a finite quotient
PID such that its ideal norm $N_{A}$ satisfies the quasi-triangle
inequality. Then the field of fractions $K$ of $A$ is a global field
and $A$ is a ring of integers of $K$.
\end{thm}

\begin{proof}
The ideal norm $N_{A}$ extends to $K$ via $N_{A}(a/b)=\frac{N_{A}(a)}{N_{A}(b)}$,
for $a,b\in A$ and it is immediately checked that $N_{A}$ on $K$
satisfies the quasi-triangle inequality. Thus $N_{A}$ on $K$ is
an absolute value. We also have $\mfp$-adic absolute values for every
nonzero prime ideal $\mfp$ of $A$. Indeed, for $a\in A$, let $v_{\mfp}(a)$
denote the largest integer $n$ such that $\mfp^{n}$ divides the
ideal $aA$, and define

\[
|a|_{\mfp}=|A/\mfp|^{-v_{\mfp}(a)}.
\]
Just like $N_{A}$, the function $|\cdot|_{\mfp}:a\mapsto|a|_{\mfp}$
extends to $K$ via $|a/b|_{\mfp}=\frac{|a|_{\mfp}}{|b|_{\mfp}}$,
and this defines an absolute value on $K$. Note that $N_{A}$ is
not equivalent to any of the absolute values $|\cdot|_{\mfp}$ because
if $p\in A$ is a generator of a prime ideal $\mfp$, we have $|p|_{\mfp}=|A/\mfp|^{-1}<1$,
while $N_{A}(p)=|A/\mfp|>1$.

We will now verify that $K$ together with the absolute values $N_{A}$
and $|\cdot|_{\mfp}$, where $\mfp$ runs through the prime ideals
of $A$, satisfies the conditions of Theorem~\ref{thm:Artin--Whaples}. 

Condition~1: Since $A$ is not a field (by the definition of basic
PID), it has a nonzero proper ideal, so the ideal norm $N_{A}$ is
not the trivial absolute value on $K$. For any nonzero $a,b\in A$,
there are only finitely many prime elements of $A$ that divide $a$
or $b$, so $|a/b|_{\mfp}=1$ for all but finitely many $\mfp$. Set
$|x|_{\infty}:=N_{A}(x)$ for $x\in K$ and let $\Sigma_{f}=\{\mfp\mid\mfp\neq(0)\text{ prime ideal of }A\}$
and $\Sigma=\Sigma_{f}\cup\{\infty\}$. Note that $\Sigma_{f}$ is
nonempty since $A$ is not a field. For a nonzero $a\in A$, let $aA=\mfp_{1}^{e_{1}}\cdots\mfp_{r}^{e_{r}}$
be the prime ideal factorization, where $e_{i}=v_{\mfp_{i}}(a)$.
Then 
\begin{align*}
\prod_{i\in\Sigma}|a|_{i} & =|\mfp_{1}^{e_{1}}|_{\mfp_{1}}\cdots|\mfp_{r}^{e_{r}}|_{\mfp_{r}}\cdot|a|_{\infty}=|A/\mfp_{1}|^{-e_{1}}\cdots|A/\mfp_{r}|^{-e_{r}}\cdot|A/aA|\\
 & =|A/\mfp_{1}|^{-e_{1}}\cdots|A/\mfp_{r}|^{-e_{r}}\cdot|A/\mfp_{1}|^{e_{1}}\cdots|A/\mfp_{r}|^{e_{r}}=1,
\end{align*}
where for the penultimate equality we have used the Chinese remainder
theorem and the fact that $|A/\mfp^{n}|=|A/\mfp|^{n}$, for any prime
$\mfp$ and $n\in\N$ (see \cite[Lemma~V.3.4]{Lorenzini}). Thus also
$\prod_{i}|a/b|_{i}=1$ for any nonzero $a,b\in A$.

Condition~2: Let $\mfp\in\Sigma_{f}$. Then $|\cdot|_{\mfp}$ is
discrete since its values are of the form $|A/\mfp|^{n}$, $n\in\Z$.
Moreover, the valuation ring $\mathcal{O}_{\mfp}$ contains $A$,
so by \cite[Theorem~2.3]{Levitz-Mott} $\mathcal{O}_{\mfp}$ is a
finite quotient domain. In particular, $\bar{k}_{\mfp}=\mathcal{O}_{\mfp}/\mfp$
is finite. Thus Theorem~\cite{Artin-Whaples} implies that $K$ is
a global field. 

By \cite[Chapter~12, Corollary~1 and Theorem~4]{Artin-Alg_Num_and_Func}
the set $\{|\cdot|_{v}\mid v\in\Sigma\}$ consists of all the nontrivial
absolute values on $K$ (up to equivalence). Since $x\in K$ lies
in $A$ if and only if $v_{\mfp}(x)\geq0$ for all $\mfp\in\Sigma_{f}$,
we have
\begin{equation}
A=\{x\in K\mid|x|_{\mfp}\leq1\text{ for all }\mfp\in\Sigma_{f}\}=\bigcap_{\mfp\in\Sigma_{f}}\cO_{\mfp}.\label{eq:A-is-val-ring}
\end{equation}
If $K$ is a number field, let $\cO_{K}$ be its ring of integers.
If $K$ is a function field, we define $\cO_{K}$ as follows. As noted
just before Theorem~\ref{thm:Artin--Whaples}, $k_{0}$ is a subfield
of any $\bar{k}_{\mfp}$, so $k_{0}$ is finite. Moreover, (\ref{eq:A-is-val-ring})
implies that $k_{0}\subset A$. Let $t\in A$ be an element such that
$t\not\in k$ (such an element exists since $A$ is not a field);
then $N_{A}(t)>1$ (otherwise $|x|_{\mfp}\leq1$ for all $\mfp\in\Sigma$,
hence $t\in k_{0}$). By the proof of \cite[Chapter~12, Theorem~3]{Artin-Alg_Num_and_Func},
$K$ is a finite extension of the field of fractions $k_{0}(t)$ of
$k_{0}[t]$. In this case, let $\cO_{K}$ denote the integral closure
of $k_{0}[t]$ in $K$.

It remains to show that in either case we have $A=\cO_{K}$. Let $A_{0}=\Z$
in case $K$ is a number field and let $A_{0}=k_{0}[t]$ otherwise.
By \cite[Corollary~5.22]{Atiyah-Macdonald}, $\cO_{K}$ is the intersection
of all the valuation rings of $K$ containing $A_{0}$, where a valuation
ring $R$ of $K$ is an integral domain with field of fractions $K$
such that $x\in K$ implies $x\in R$ or $x^{-1}\in R$. It is clear
that every $\cO_{\mfp}$ is a valuation ring of $K$ containing $A_{0}$,
so that by (\ref{eq:A-is-val-ring}) we have $\cO_{K}\subseteq A$.
Conversely, we claim that every valuation ring $R$ of $K$ containing
$A_{0}$ equals some $\cO_{\mfp}$. Indeed, let $R$ be a valuation
ring of $K$ containing $A_{0}$. Then $R$ is integrally closed \cite[Proposition~5.18]{Atiyah-Macdonald},
so $\cO_{K}\subseteq R$. If $\mfm$ is the maximal ideal of $R$,
then $\mfq:=\cO_{K}\cap\mfm$ is a prime ideal of $\cO_{K}$, and
as $R$ is a local ring \cite[Proposition~5.18]{Atiyah-Macdonald},
we have $\cO_{K,\mfq}\subseteq R$, where $\cO_{K,\mfq}$ is the localization
of $\cO_{K}$ at the prime ideal $\mfq=\cO_{K}\cap\mfp$. Since $\cO_{K}$
is a Dedekind domain, $\cO_{K,\mfq}$ is a discrete valuation ring,
hence a valuation ring, so by \cite[Theorem~5.21]{Atiyah-Macdonald}
we must have $\cO_{K,\mfq}=R$. Now, since $A$ is a PID it is integrally
closed (and $A$ contains $A_{0}$), so we must have $\cO_{K}\subseteq A$.
Let $\mfp$ be a prime ideal of $A$ such that $\mfq=\cO_{K}\cap\mfp$
(i.e., $\mfp$ can be any prime ideal dividing the ideal $\mfq A$).
Then $\cO_{K}\subseteq A\subseteq\cO_{\mfp}$, so $\cO_{K,\mfq}\subseteq\cO_{\mfp}$
and by \cite[Theorem~5.21]{Atiyah-Macdonald} $\cO_{K,\mfq}=\cO_{\mfp}$
and thus $R=\cO_{\mfp}$. It thus follows from (\ref{eq:A-is-val-ring})
and \cite[Corollary~5.22]{Atiyah-Macdonald} that $A=\cO_{K}$.
\end{proof}
Let $B$ be a ring of class (G) over the basic PID $A$, let $K$
be the fraction field of $A$, and let $L$ be the fraction field
of $B$. With this notation, we have the following result.
\begin{lem}
\label{lem:Field-extension}The field extension $L/K$ is of finite
degree and $B$ is the integral closure of $A$ in $L$. On the other
hand, let $L'/K$ be a finite separable extension and let $B'$ be
the integral closure of $A$ in $L'$. Then $B'$ is a ring of class
(G) over $A$.
\end{lem}

\begin{proof}
Let $S=A\setminus\{0\}$. Then (by a simple argument) $S^{-1}B$ is
finitely generated as a vector space over $S^{-1}A=K$. Thus $S^{-1}B$
is an integral domain that is a finite-dimensional vector space, so
$S^{-1}B$ is a field, that is, $S^{-1}B=L$. Hence $L/K$ is finite.
Furthermore, since $B$ is finitely generated over $A$, $B$ is integral
over $A$ (see, e.g., \cite[Proposition~I.2.10]{Lorenzini}). Thus
$B$ lies inside the integral closure $C$ of $A$ in $L$. Since
$B\subseteq C\subseteq L$, the fraction field of $C$ is $L$. Any
$x\in C$ is integral over $A$, hence integral over $B$. Since $B$
is a Dedekind domain it is integrally closed, so $x\in B$. Thus $C=B$,
that is, $B$ is the integral closure of $A$ in $L$.

Moreover, it is well known that $B'$ is a Dedekind domain that is
finitely generated over $A$ (see, e.g., \cite[Theorem~I.4.7 and Theorem~I.6.2]{Lorenzini}).
Since $B'$ is torsion-free, it is free over $A$ and thus $B'$ is
a ring of class (G) over $A$.
\end{proof}
\begin{cor}
\label{cor:class-(G)-is-global-rings}Let $A$ be a finite quotient
PID such that its ideal norm $N_{A}$ satisfies the quasi-triangle
inequality. Let $B$ be a Dedekind domain which is a finitely generated
and free $A$-module. Then $B$ is a ring of integers of a global
field. In particular, if $A$ is a basic PID and $B$ is of class
(G) over $A$, then $B$ is a ring of integers of a global field.
\end{cor}

\begin{proof}
Let $K$ and $L$ be the fraction field of $A$ and $B$, respectively.
By Lemma~\ref{lem:Field-extension} $L$ is a global field and $B$
is the integral closure of $A$ in $L$. By Proposition~\ref{thm:BasicPIDs-are-rings-of-integers},
$A$ is the integral closure in $K$ of $A_{0}$, where $A_{0}$ is
either $\Z$ or $\F_{q}[t]$, for some $t\in K$. Let $C$ be the
integral closure of $A_{0}$ in $L$. Since $A_{0}\subseteq A$ we
trivially have $C\subseteq B$. By the transitivity of integrality
\cite[Proposition~I.2.18]{Lorenzini} applied to $A_{0}\subseteq A\subseteq B$,
we have that $B$ is integral over $A_{0}$, hence $B\subseteq C$
and so $B=C$. We have proved that $B$ is the integral closure of
$\Z$ or $\F_{q}[t]$ in the global field $L$ and thus $B$ is a
ring of integers in $L$.
\end{proof}
One may ask whether there exists a Dedekind domain $B$ that is finitely
generated and free over a PID $A$ with finite quotients and such
that $B$ has infinite class group. Theorem~\ref{thm:Bound-finiteness}
and Corollary~\ref{cor:class-(G)-is-global-rings} show that if such
an example exists, then the quasi-triangle inequality must fail for
ideal norm $N_{A}$. We note that Goldman \cite{Goldman} and Heitmann
\cite{Heitmann} have given examples of Dedekind domains with finite
quotients and infinite class groups, but we do not know whether these
examples are finitely generated and free over some PID.

It is a trivial fact that there exist Dedekind domains (even PIDs)
with finite class groups that are not overrings of any ring of integers
of a global field. Indeed, the polynomial ring $\C[X]$ is a PID but
is not a finite quotient domain, so cannot be an overring of any finite
quotient domain (finite quotient domains are stable under localization).
However, we do not know whether there exists a finite quotient Dedekind
domain with finite class group that is not the overring of any ring
of integers of a global field. 
\begin{acknowledgement*}
I wish to thank Pete L.~Clark and D.~Lorenzini for pointing out
inaccuracies in a previous version and for comments that helped to
improve the exposition.
\end{acknowledgement*}
\bibliographystyle{alex}
\bibliography{alex}

\end{document}